\pdfoutput=1

\documentclass{amsart}

\usepackage{verbatim}
\usepackage{xcolor}
\usepackage{tikz}
\usepackage{tikz-cd}
\usepackage{arydshln}
\usepackage{caption}
\usepackage{subcaption}
\usepackage{graphicx}
\usepackage[hyphens]{url}
\usepackage{amsmath}
\usepackage{bbm}
\usepackage{hyperref}
\usepackage{mathrsfs}
\usepackage{todonotes}
\usepackage[T1]{fontenc}

\newtheorem{theorem}{Theorem}[section]
\newtheorem{proposition}[theorem]{Proposition}
\newtheorem{lemma}[theorem]{Lemma}
\newtheorem{corollary}[theorem]{Corollary}
\newtheorem{claim}[theorem]{Claim}

\theoremstyle{definition}

\theoremstyle{remark}
\newtheorem{remark}[theorem]{Remark}
\newtheorem{question}[theorem]{Question}
\newtheorem{problem}[theorem]{Problem}

\numberwithin{equation}{section}

\newcommand{\cX}{\mathcal{X}}
\newcommand{\bR}{\mathbf{R}}

\newcommand{\mcg}{\mathrm{Mod}}

\newcommand{\qut}{\mathcal{Q}^1\mathcal{T}}

\newcommand{\qum}{\mathcal{Q}^1\mathcal{M}}
\newcommand{\ml}{\mathcal{ML}}

\newcommand{\pml}{\mathcal{PML}}

\renewcommand{\put}{\mathcal{P}^1\mathcal{T}}
\newcommand{\pum}{\mathcal{P}^1\mathcal{M}}

\newcommand{\SL}{\mathrm{SL}}

\newcommand{\Aut}{\mathrm{Aut}}

\renewcommand{\bold}[1]{\medskip \noindent {\bf #1 }\nopagebreak}

\begin{document}

\title[The asymmetry of Thurston's earthquake flow]{The asymmetry of Thurston's earthquake flow}

\author{Francisco Arana--Herrera}
%\address{Department of Mathematics, Stanford University, 450 Jane Stanford Way, Stanford, CA 94305, USA.}
%\email{farana@stanford.edu}

\author{Alex Wright}
%\address{Department of Mathematics, University of Michigan, 530 Church Street, Ann Arbor, MI 48109, USA.}
%\email{alexmw@umich.edu}

\begin{abstract}
We show that Thurston's earthquake flow is strongly asymmetric in the sense that its normalizer is as small as possible inside the group of orbifold automorphisms of the bundle of measured geodesic laminations over moduli space. (At the level of Teichm\"uller space, such automorphisms correspond to homeomorphisms that are equivariant with respect to an automorphism of the mapping class group.) It follows that the earthquake flow does not extend to an $\SL(2, \mathbf{R})$-action of orbifold automorphisms and does not admit continuous renormalization self-symmetries. In particular, it is not conjugate to the Teichm\"uller horocycle flow via an orbifold map.  This contrasts with a number of previous results, most notably Mirzakhani's theorem that the earthquake and Teichm\"uller horocycle flows are measurably conjugate.
\end{abstract}

\maketitle


\thispagestyle{empty}

\tableofcontents
\section{Introduction}

\subsection*{Context.} The bundle $\pum_g$ of unit length measured geodesic laminations  over the moduli space $\mathcal{M}_g$ of hyperbolic or Riemann surfaces of genus $g$ is most naturally seen as a construction of hyperbolic geometry, whereas the bundle $\qum_g$ of unit area quadratic differentials over $\mathcal{M}_g$ is most naturally seen from the either the perspective of complex analysis or flat geometry. The bundle $\pum_g$ supports Thurston's rather mysterious earthquake flow, which is most concisely defined as a Hamiltonian flow using the Weil-Petersson symplectic form, whereas the bundle $\qum_g$ supports the Teichm\"uller horocycle flow, easily defined as part of the much studied $\mathrm{SL}(2, \mathbf{R})$-action. Mirzakhani showed that, despite their different origins, these flows are measurably isomorphic \cite{Mir08a}.  

\begin{theorem}[Mirzakhani]
    \label{theo:MM}
    There is a measurable conjugacy $ \pum_g \to \qum_g$ between the earthquake flow and the Teichm\"uller horocycle flow.
\end{theorem}

In addition to being of fundamental interest as a bridge between different perspectives on the geometry of surfaces and their moduli spaces, this theorem has powered many applications concerning  equidistribution, counting, and the study of random surfaces \cite{Mir07b, Ara19b, Mir16, Ara20a, Liu19, lu2021counting}. 

%For an expository account of why Theorem \ref{theo:MM} is true, see \cite{Wri18}, and for a introduction to Mirzakhani's work more broadly, see for example \cite{Wri19}.

Mirzakhani's conjugacy is only defined on a full measure subset of $\pum_g$, and, as remarked by Mirzakhani herself \cite[\S 6]{Mir08a}, this conjugacy cannot be extended to a continuous map on all $\pum_g$. Despite this, recent work of Calderon and Farre extended Mirzakhani's conjugacy to a bijection which, although not continuous, is geometrically natural and has exciting new applications \cite{calderon2021shear}.  

One reason Theorem \ref{theo:MM} is plausible is that there are many  conceptual similarities between the earthquake flow and the Teichm\"uller horocycle flow, such as the following: 
\begin{enumerate}
\setlength{\itemsep}{4pt}
\item Both arise from some notion of shearing. 
\item Both have been understood by analogy to unipotent flows on homogeneous spaces.
\item Both are Hamiltonian with respect to related symplectic structures \cite{M95, BS01}. 
\item Both are associated to natural complex disks in Teichm\"uller space, namely Teichm\"uller discs for the Teichm\"uller horocycle flow and complex earthquake discs for the earthquake flow \cite{Complex}. 
\item Both have quantitative non-divergence properties \cite{MW02}. 
\end{enumerate}

\subsection*{No continuous conjugacy.} In light of all these similarities and the work of Mirzakhani, Calderon and Farre, one might wonder if a result stronger than Theorem \ref{theo:MM} holds: perhaps the earthquake and Teichm\"uller horocycle flows are isomorphic from the point of view of continuous dynamics, i.e., perhaps there is a different conjugacy between these flows that is also a homeomorphism. This question was advertised in \cite[Problem 12.3]{Wri19} and \cite[Remark 5.6]{Wri18}. Our main result on asymmetry, which we will state shortly as Theorem \ref{theo:newmain}, implies a negative solution to this problem.

\begin{theorem}
    \label{theo:main}
    There does not exist an orbifold conjugacy $\pum_g \to \qum_g$ between the earthquake flow and the Teichmüller horocycle flow.
\end{theorem}

The technical restriction in Theorem \ref{theo:main} that the conjugacy respects the orbifold structure of these spaces is natural since both spaces have the same orbifold structure \cite{HM79}. %(Indeed, by work of Hubbard and Masur \cite{HM79}, the map $\smash{\widehat{H}} \colon \pum_g \to \qum_g$ that to any unit length measured geodesic lamination on a genus $g$ hyperbolic surface assigns the unique unit area quadratic differential on the surface whose vertical foliation corresponds to the measured geodesic lamination is an isomorphism of orbifolds.) 

The existence of an orbifold conjugacy $\pum_g \to \qum_g$ as in Theorem \ref{theo:main} is equivalent to the existence of a topological conjugacy $\put_g \to \qut_g$ of the lifts to Teichmüller space of the earthquake and Teichmüller horocycle flows that intertwines an automorphism $\rho \colon \mcg_g \to \mcg_g$ of the  mapping class group. For detailed discussions on the theory of orbifolds see \cite[Chapter 13]{T80} and \cite[\S 2]{ES20b}. In particular, the following corollary holds.

\begin{corollary}
    \label{cor:main}
    There does not exist a mapping class group equivariant topological conjugacy $\put_g \to \qut_g$ between 
    %the lifts to Teichmüller space of 
    the earthquake flow and the Teichmüller horocycle flow.
\end{corollary}

\subsection*{Strong asymmetry.} A flow $E= \{E_t \colon \cX \to \cX\}_{t \in \mathbf{R}}$ on a space $\cX$ can be interpreted as a group homomorphism $E\colon \mathbf{R} \to \Aut(\cX)$ mapping $t \in \mathbf{R}$ to $E_t \in \Aut(\cX)$, where the automorphism group $\Aut(\cX)$ is defined in whatever category (smooth, continuous, measurable, etc.) is under consideration. 

The centralizer of the flow $E$ is defined as 
$$C(E) = \{ F \in \Aut(\cX) \colon \forall_{t\in \mathbf{R}} \ E_t \circ F = F \circ  E_t \}.$$
The centralizer corresponds to the most narrow concept of the set of symmetries of a flow one can consider, consisting only of the automorphisms that commute with it. A slightly broader notion is the extended centralizer a flow, defined here as 
$$C_\pm(E) = \{ F \in \Aut(\cX) \colon \exists_{\varepsilon \in \{1,-1\}}  \forall_{t\in \mathbf{R}} \ E_t \circ F = F \circ E_{\varepsilon t} \}.$$
The extended centralizer includes time reversing symmetries of a flow.

Even more broadly, one can consider the normalizer of a flow, defined  as
$$N(E) = \{ F \in \Aut(\cX) : \exists_{\varepsilon \in \{1,-1\}, s \in \bR}  \forall_{t\in \mathbf{R}} \ E_t \circ F = F \circ E_{\varepsilon e^{2s} t} \}.$$
The normalizer includes symmetries that scale time, i.e, which conjugate the flow to a constant speed reparametrization of itself. If $F\in N(E)$ is as above, we call $F$ a \emph{normalizer} of the flow, or an \emph{$s$-normalizer}  if we wish to specify the time dilation factor $e^{2s}$. 

The smallest $N(E)$ can be is the flow itself, namely $N(E)= \{E_t \}_{t\in \mathbf{R}}$. When this is the case, we say that the flow $E$ is \textit{strongly asymmetric}. Our main result establishes this strong asymmetry property for the earthquake flow. 

\begin{theorem}
    \label{theo:newmain}
    The normalizer of the earthquake flow inside the group of orbifold automorphisms of $\pum_g$ is the flow itself.
\end{theorem}

Theorem \ref{theo:main} follows immediately from Theorem \ref{theo:newmain}, since the Teichm\"uller horocycle flow is normalized by the Teichm\"uller geodesic flow. 

\bold{A few remarks.} Before discussing the proof of Theorem \ref{theo:newmain}, let us make a couple of remarks.

\begin{remark} In testing the plausibility of Theorem \ref{theo:newmain}, it is natural to consider both Thurston's stretch map flow, defined in \cite{Thu98}, and grafting, so we discuss both in turn. 

The stretch map flow already makes a natural appearance in any discussion regarding Mirzakhani's conjugacy. Indeed, Mirzakhani's conjugacy shows that the earthquake flow is part of a measurable $\SL(2,\mathbf{R})$-action in which the diagonal subgroup acts via the stretch map flow. The stretch map flow does normalize the earthquake flow, but, since it is only defined on a full measure subset of $\pum_g$, this does not contradict Theorem \ref{theo:newmain}. 

Grafting plays a central role in the definition of complex earthquake discs. If one compares Teichm\"uller discs to complex earthquake discs, the Teichm\"uller geodesic flow corresponds to grafting. Grafting is continuous, but, since it does not normalize the earthquake flow, this does not contradict Theorem \ref{theo:newmain}. 
\end{remark}

In the next two remarks, it is implicit that we are working in the category of topological orbifolds (so in particular all conjugacies are continuous). 

\begin{remark}
Theorem \ref{theo:newmain} implies that the earthquake flow is not conjugate to its own inverse. (The inverse of a flow $t  \mapsto E_t $ is the flow $t\mapsto E_{-t}$.)
\end{remark}

\begin{remark}
Theorem \ref{theo:newmain} implies that the earthquake flow is not the restriction of any $\SL(2,\bR)$-action to any one-parameter subgroup. (One way to see this is to note that every non-compact one-parameter subgroup of $\SL(2,\bR)$ has non-trivial normalizer, since the horocycle flow is normalized by the geodesic flow and the geodesic flow is normalized by an involution.)
\end{remark}

\subsection*{Outline of the proof.}  Every normalizer can and should be considered as a conjugacy between the earthquake flow and a (possibly trivial) linear time change of itself. Given an $s$-normalizer $F \colon \pum_g \to \pum_g$, we constrain its behavior until we are eventually able to show it is an element of the flow. This involves four main steps, each occupying a different section of this paper. Throughout we assume $(X,\lambda)\in \pum_g$ and $F(X,\lambda)=(Y,\mu)$. 
\begin{enumerate}
\setlength{\itemsep}{7pt}
\item By studying minimal sets, we show in Proposition \ref{prop:step_2} that $\mu$ is a multi-curve if and only if $\lambda$ is, and, moreover,  that the number of components of $\mu$ is equal to the number of of components of $\lambda$. This is strongly related to work of Minsky, Smillie, and Weiss \cite{MW02, SW04}.
\item Leveraging the rigidity of the curve complex, we show in Proposition \ref{prop:step_4} that $\mu$ is a multiple of $\lambda$. This relies on  work of  Ivanov \cite{I97} and applies to all $(X,\lambda)\in \pum_g$. 
\item By carefully analysing the periods of specific closed orbits, we determine  in Lemma \ref{lem:step_5} what the multiple is, showing $\mu = e^s \cdot \lambda$. We moreover show in Lemma \ref{lem:step_6} that, often, many curves shrink by at least a factor of $e^{-s}$ in the passage of $X$ to $Y$. This gives a contradiction unless $s=0$, showing that the normalizer is equal to the extended centralizer, a conclusion we record as Proposition \ref{prop:nor}.  
\item In Proposition \ref{prop:extendedcentralizer}, we show that the extended centralizer of the earthquake flow is trivial, by showing that many and hence all orbits are preserved, and using ergodicity. We use Mirzakhani's generalized McShane identity \cite{Mir07a} as a technical tool. 
\end{enumerate}

\subsection*{Open problems.} Many interesting questions related to Mirzakhani's conjugacy remain open. We highlight a few of them here. 

To our knowledge, the only previously established dynamical difference between the earthquake and Teichm\"uller horocycle flows concerns cusp excursions in the specific case of  once punctured tori \cite{fu2019cusp}. Previous to this, it was known that certain orbits of the two flows do not stay finite distance apart in one dimensional Teichm\"uller spaces \cite[Proposition 8.1]{MW02}. 

Theorem \ref{theo:newmain} is a dynamical difference, since it relates to renormalization, but it would be illuminating to find less subtle differences. 

\begin{problem}
Find a dynamical, non-group theoretic property that is invariant under topological conjugacies and which holds for exactly one of the earthquake flow and the Teichmüller horocycle flow.
\end{problem}

It is easy to construct topological joinings between the earthquake flow and the Teichmüller horocycle flow. For example, consider the set of pairs $$( (X, \lambda), q) \in \pum_g \times \qum_g$$ such that the horizontal foliation of $q$ is equal to $\lambda$. This construction of a topological joining admits many different variations.

\begin{problem}
Classify all the topological joinings between the earthquake flow and the Teichmüller horocycle flow. 
\end{problem}

More generally, our dynamical understanding of the earthquake flow remains incomplete, leaving questions such as the following open. 

\begin{question}
Is the earthquake flow polynomially mixing?
\end{question}

In comparison, it is known that the Teichm\"uller horocycle flow is polynomially mixing \cite{AGY06, AR12, AG13, Ra87}.

There are also interesting open questions related to strong asymmetry, including the following deliberately vague question. 

\begin{question}
How common is strong asymmetry in smooth dynamics?
\end{question}

The most interesting setting for this question may be flows that share some properties with the earthquake flow, such as volume preserving flows with zero entropy and having closed orbits of all periods. 

Centralizers of flows (and diffeomorphisms) have been studied, for example, in \cite{Obata, BF, BV}. Actions of {B}aumslag-{S}olitar groups and other discrete solvable groups have been studied, for example, in \cite{BMN, GLcirc, GLsurf, WX, BW, McCarthy}. Actions of solvable Lie groups have been studied, for example, in \cite{Ghys, GV}. See the ICM notes of Wilkinson \cite{WilkinsonICM} and Navas \cite{NavasICM} for some open questions and additional context. See the book \cite{Navas} for more on the one-dimensional case.  

In \cite{FKPL} a continuous flow on the torus is constructed that (in particular) has a measurable $s$-normalizer for every $s\in \bR$ but has no continuous $s$-normalizers for $s\neq 0$. In light of Theorem \ref{theo:newmain} and the work of Mirzakhani, this is analogous to the situation for the earthquake flow. In \cite{FL}, the symmetries of certain flows are studies in the measurable category. In \cite{DisjointFromInverse}, time reversing translation flows are studied. 

%\color{red}
%\begin{itemize}
%\item In \cite[Corollary 8.7]{FKPL}, a continuous flow on the torus is constructed that (in particular) has a measurable $s$-normalizer for every $s\in \bR$ but has no continuous $s$-normalizers for $s\neq 0$. (They use the terminology ``self-similarity" rather than ``symmetry".) In light of Theorem \ref{theo:newmain} and the work of Mirzakhani, this is analogous to the situation for earthquake flow. 
%\item In \cite{FL}, the symmetries of certain flows are studies in the measurable category. 
%\item In \cite{DisjointFromInverse}, time reversing translation flows are studied. 
%\item In \cite{Marcus}, ???
%\end{itemize}
%\color{black}

% Things to mention:
% \begin{itemize}
%     \item Same orbifold structure.
%     \item Easy to construct joinings.
%     \item Theorem might generalize to conjugacies between subloci.
% \end{itemize}

\subsection*{Acknowledgements.} The first author is very grateful to Steve Kerckhoff for his invaluable advice, patience, and encouragement. The authors are grateful to Giovanni Forni, Krzysztof Fr\k{a}czek, Corinna Ulcigrai, and  Amie Wilkinson for enlightening conversations on previous work in dynamics, and Aaron Calderon and Barak Weiss for helpful conversations regarding the appendix.  This work was finished while the first author was a member of the Institute for Advanced Study (IAS). The first author is very grateful to the IAS for its hospitality. This material is based upon work supported by the National Science Foundation under Grant No. DMS-1926686. During the preparation of this paper the second author was partially supported by  NSF Grant DMS 1856155 and a Sloan Research Fellowship.

\section{A dimension argument using minimal sets}

In this section we analyze minimal sets to obtain the following. 

\begin{proposition}
    \label{prop:step_2}
    Let $F \colon \pum_g \to \pum_g$ be a normalizer of the earthquake flow, and suppose $(X,\lambda)\in \pum_g$ and  $F(X,\lambda)=(Y,\mu)$. Then, for any $k \in \mathbf{N}$, $\lambda$ is a simple closed multi-curve with $k$ components if and only if $\mu$ is a simple closed multi-curve with $k$ components.
\end{proposition}

We begin by showing that every normalizer must preserve the locus of points  $(X,\lambda) \in \pum_g$ with $\lambda$ a simple closed multi-curve. We do this using the minimal sets of the earthquake flow. 

A minimal set of the earthquake flow is a closed, earthquake flow invariant subset of $\pum_g$ that does not contain any proper, non-empty, closed, earthquake flow invariant subsets. % We say a point $(X,\lambda) \in \pum_g$ is bounded if it belongs to a minimal set of the earthquake flow. 

We will be interested in compact minimal sets. Minsky and Weiss showed that all minimal sets for the earthquake flow are compact \cite{MW02}, but we will not require such a strong statement. The result we will need is the following.

\begin{theorem}
    \label{theo:min}
    A point $(X,\lambda) \in \pum_g$ is contained in a compact minimal set if and only if $\lambda$ is a simple closed multi-curve.
\end{theorem}

Smillie-Weiss  \cite{SW04} proves the analogous statement for the Teich\"muller horocycle flow and states that it should be possible to similarly obtain a result for the earthquake flow. However as far as we know even the statement of Theorem \ref{theo:min} has not previously appeared in the literature. 
For the convenience of the reader we sketch a proof in Appendix \ref{A:SW}. 
 
Since normalizers preserve minimal sets, we  deduce the following corollary.

\begin{corollary}
    \label{cor:step_1}
    Let $F \colon \pum_g \to \pum_g$ be a normalizer of the earthquake flow, and suppose $(X,\lambda)\in \pum_g$ and  $F(X,\lambda)=(Y,\mu)$. Then, $\lambda$ is a simple closed multi-curve if and only if $\mu$ is a simple closed multi-curve.
\end{corollary}

To get a grasp on the number of components of a simple closed multi-curve, we study the local topology of the lift to $\put_g$ of the union of the compact minimal sets of the earthquake flow on $\pum_g$. The following result is crucial to our approach.

\begin{lemma}
    \label{lem:dim}
    Let $\gamma \in \pml_g$ be the projective class of a simple closed multi-curve with $k \in \mathbf{N}$ components, $U \subseteq \pml_g$ be a small open neighborhood of $\gamma$ in $\pml_g$, and $W$ be the path connected component containing $\gamma$ of the intersection of $U$ with the subset of $\pml_g$ of projective classes of simple closed multi-curves. Then, if $U$ is sufficiently small, $U \cap \overline{W}$ is locally homeomorphic to $\mathbf{R}^{6g-7-k}$.
\end{lemma}

\begin{proof}
Denote $\gamma := \smash{\sum_{i=1}^k a_i \gamma_i} \in \pml_g$. Then, if $U$ is sufficiently small, $W$ consists of projective classes of simple closed multi-curves of the form 
\[
\gamma' := \sum_{i=1}^k (a_i+\epsilon_i) \gamma_i + \sum_{j=1}^{k'} \delta_j \gamma_j',
\]
where $\epsilon := \smash{(\epsilon_i)_{i=1}^k} \in \mathbf{R}^k$ is a small vector, $k' \geq 0$ is a non-negative integer, $\smash{(\gamma_j')_{j=1}^{k'}}$ are pairwise non-homotopic and non-intersecting simple closed curves that are not homotopic and do not intersect any of the components of $\gamma$, and $\delta := \smash{(\delta_j)_{j=1}^{k'}} \in \mathbf{R}_+^{k'}$ is a small vector with positive entries. This fact can be readily verified using Dehn-Thurston coordinates \cite[\S1.2]{PH92}. Indeed, if $U$ is sufficiently small, projective classes in $W$ correspond to simple closed multi-curves whose geometric intersection number with any of the components of $\gamma$ is zero. 

Furthermore, the closure of $W$ in $U$ is given by the connected component containing $\gamma$ of the intersection of $U$ with the projectivization of
\[
\mathcal{Z}_g(\gamma) := \{\lambda \in \ml_g \ | \ i(\gamma,\lambda ) = 0\}.
\]
Notice that $\mathcal{Z}_g(\gamma)$ is homeomorphic to $\mathbf{R}^k \times \mathbf{R}^{6g-6-2k}$, where the first term of this product corresponds to changing the weights of the components of $\gamma$ and the second term corresponds to choosing a measured geodesic lamination on $S_g$ supported away from $\gamma$. In particular, $U \cap \overline{W}$ is locally homeomorphic to $\mathbf{R}^{6g-7-k}$.
\end{proof}

%We can now conclude are argument. 

%\begin{proof}[Proof of Proposition \ref{prop:step_2}]
Suppose $(X,\gamma) \in \put_g$, where $\gamma$ is a simple closed multi-curve with $k \in \mathbf{N}$ components. Consider a small open neighborhood $U \subseteq \put_g$ of $(X,\gamma)$. Denote by $W$ the path connected component containing $(X,\gamma)$ of the intersection of $U$ with the subset of points of $\put_g$ where the lamination is a simple closed multi-curve. Directly from Lemma \ref{lem:dim} we see that, if $U$ is sufficiently small, $U \cap \smash{\overline{W}}$ is locally homeomorphic to $\mathbf{R}^{12g-13-k}$; the $6g-6$ increase in dimension with respect to Lemma \ref{lem:dim} comes from the dimension of Teichmüller space. In particular, we can recover the number of components of $\gamma$ from the dimension of this intersection.

As the number of components of $\gamma$ can be recovered from information depending exclusively on the minimal sets of $\pum_g$, 
this quantity is preserved by any earthquake flow normalizer. This concludes the proof of Proposition \ref{prop:step_2}.
%\end{proof}

\section{An automorphism of the curve complex}

In this section we use the rigidity of the curve complex to obtain the following. 

\begin{proposition}
    \label{prop:step_4}
    Every normalizer $F \colon \pum_g \to \pum_g$ of the earthquake flow admits a $\mcg_g$-equivariant lift $\smash{\widehat{F}} \colon \put_g \to \put_g$ such that for every $(X,\lambda) \in \put_g$, if $\smash{\widehat{F}}(X,\lambda)=(Y,\mu)$, then $\mu$ belongs to the projective class of $\lambda \in \ml_g$.
\end{proposition}

Because we assume $F$ is an orbifold map, there exists a lift $\smash{\widehat{F}} \colon \put_g \to \put_g$ that is equivariant with respect to some automorphism of $\mcg_g$. We start with this lift and show how to modify it to get the desired lift $\smash{\widehat{F}}$. 

 Denote by $\mathcal{S}_g$ the discrete set of free homotopy classes of unoriented simple closed curves on the marking surface $S_g$. By Proposition \ref{prop:step_2}, every $X \in \mathcal{T}_g$ induces a map $\Psi_X \colon \mathcal{S}_g \to \mathcal{S}_g$ in the following way: Given $\gamma \in \mathcal{S}_g$, let $\Psi_X(\gamma) \in \mathcal{S}_g$ be the free homotopy class of the simple closed curves $\gamma'$ given by $$(Y,\gamma'/\ell_{\gamma'}(Y)) := \smash{\widehat{F}}(X,\gamma/\ell_\gamma(X)).$$
As $\mathcal{T}_g$ is connected and as $\mathcal{S}_g$ is discrete, the map $\Psi_X \colon \mathcal{S}_g \to \mathcal{S}_g$ is independent of $X \in \mathcal{T}_g$. From now on we denote this map simply by $\Psi \colon \mathcal{S}_g \to \mathcal{S}_g$. 

We claim that $\Psi$ induces an automorphism of the curve complex of $S_g$, meaning that it is bijective and that any pair of simple closed curves can be realized disjointly if and only if their images under $\Psi$ can be realized disjointly. 

\begin{lemma}
    \label{lem:step_3}
    The map $\Psi \colon \mathcal{S}_g \to \mathcal{S}_g$ defined above induces an automorphism of the curve complex of $S_g$.
\end{lemma}

\begin{proof}
    An inverse of $\Psi \colon \mathcal{S}_g \to \mathcal{S}_g$ can be constructed using the inverse of $\smash{\widehat{F}}$. It follows that $\Psi$ is bijective. 
    
    Notice that a pair $\alpha,\beta \in \mathcal{S}_g$ of simple closed curves  can be realized disjointly if and only if there exists a path $$t \in [0,1] \mapsto (X_t,\gamma_t) \in \put_g$$ such that $\gamma_t$ is a simple closed multi-curve on $S_g$ for every $t \in [0,1]$, $\gamma_0 = \alpha/\ell_{X_0}(\alpha)$, $\gamma_1 = \beta/\ell_{X_1}(\beta)$, and $\gamma_t$ has exactly two components for every $t \in (0,1)$. It follows from Proposition \ref{prop:step_2} that $\smash{\widehat{F}}$ preserves these types of paths. In particular, for every pair of simple closed curves $\alpha,\beta \in \mathcal{S}_g$, their images $\Psi(\alpha),\Psi(\beta) \in \mathcal{S}_g$ are non-intersecting if and only if $\alpha$ and $\beta$ are non-intersecting.
\end{proof}

A well known result of Ivanov \cite{I97} shows that every automorphism of the curve complex of a closed, connected, oriented surface $S_g$ of genus $g \geq 2$ is induced by the isotopy class of a diffeomorphism of $S_g$. Thus there exists a diffeomorphism $\psi \colon S_g \to S_g$ such that the map $\Psi \colon \mathcal{S}_g \to \mathcal{S}_g$ defined above is given by $\Psi(\gamma) = \psi(\gamma)$ for every $\gamma \in \mathcal{S}_g$. The diffeomorphism $\psi$ acts on $\put_g$ by changing the markings even if it does not preserve the orientation of $S_g$. It also acts naturally on the mapping class group $\mcg_g$ by conjugation.

Since $\smash{\widehat{F}} \colon \put_g \to \put_g$ is the lift of an orbifold map, there exists an automorphism $\rho \colon \mcg_g \to \mcg_g$ 
such that $$\smash{\widehat{F}}(\phi.(X,\lambda)) = \rho(\phi).\smash{\widehat{F}}(X,\lambda)$$
 for every $\phi \in \mcg_g$ and every $(X,\lambda) \in \put_g$. Consider the lift $\smash{\widehat{F}}' \colon \put_g \to \put_g$ of $F$ defined by 
 $$\smash{\widehat{F}}'(X,\lambda) := \psi^{-1}.\smash{\widehat{F}}(X,\lambda).$$
 This lift intertwines the automorphism $\rho' \colon \mcg_g \to \mcg_g$ given by $$\rho'(\phi) := \psi^{-1} \circ \rho(\phi) \circ \psi$$ for every $\phi \in \mcg$. Thus, by replacing $\smash{\widehat{F}}$ with $\smash{\widehat{F}}'$, we can assume without loss of generality that the map $\Psi \colon \mathcal{S}_g \to \mathcal{S}_g$ defined above is the identity.

As $\smash{\widehat{F}}$ intertwines the automorphism $\rho \colon \mcg_g \to \mcg_g$, the map $\Psi \colon \mathcal{S}_g \to \mathcal{S}_g$ defined above, which we are assuming is the identity, also intertwines this automorphism. It follows that $\rho(\phi).\gamma = \phi.\gamma$ for every $\phi \in \mcg_g$ and every $\gamma \in \mathcal{S}_g$. As the kernels of the $\mcg_g$ actions on $\mathcal{S}_g$ and $\mathcal{T}_g$ are equal, $\rho(\phi).X = X$ for every $\phi \in \mcg_g$ and every $X \in \mathcal{T}_g$. It follows that, without loss of generality, we can assume that the automorphism $\rho \colon \mcg_g \to \mcg_g$ is the identity.

The discussion above shows that the lift $\smash{\widehat{F}}$ satisfies the following property: For every $X \in \mathcal{T}_g$ and every simple closed curve $\gamma \in \mathcal{S}_g$, if $(Y,\mu) := \smash{\widehat{F}}(X,\gamma/\ell_\gamma(X))\in \put_g$, then $\mu$ belongs to the projective class of $\gamma \in \ml_g$. As simple closed curves are dense in $\pml_g$, the same property holds for arbitrary measured geodesic laminations. This concludes the proof of Proposition \ref{prop:step_4}.

\section{Inspecting the periods of closed orbits}

In this section we show that the normalizer of the earthquake flow is equal to its extended centralizer. 

\begin{proposition}
    \label{prop:nor}
    $N(E)=C_\pm(E)$. 
\end{proposition}

In other words, given an $s$-normalizer $F$ as above, we show that $s=0$. We begin by strengthening Proposition \ref{prop:step_4} to control the scaling between $\lambda$ and $\mu$. %Denote by $\smash{\widehat{E}} := \{\smash{\widehat{E}_t} \colon \pum_g \to \pum_g\}_{t \in \mathbf{R}}$ the earthquake flow on $\pum_g$. Given $s \in \mathbf{R}$, we say an orbifold isomorphism $\smash{\widehat{G}_s} \colon \pum_g \to \pum_g$ is an $s$-normalizer of the earthquake flow if $\smash{\widehat{E}_t} \circ \smash{\widehat{G}_s} = \smash{\widehat{G}_s} \circ \smash{\widehat{E}_{e^{2s}t}}$ for every $t \in \mathbf{R}$.

\begin{lemma}
    \label{lem:step_5}
    Let $\smash{\widehat{F}}$ be the lift produced by Proposition \ref{prop:step_4} of an $s$-normalizer $F$. Then, for every $(X,\lambda) \in \put_g$, if $(Y,\mu) := \hat{F}(X,\lambda)$, then $\mu = e^s \cdot \lambda$.
\end{lemma}

\begin{proof}
    %Let $\smash{\widehat{G}_s} \colon \pum_g \to \pum_g$ be an $s$-normalizer of the earthquake flow and $G_s \colon \put_g \to \qut_g$ be the $\mcg_g$-equivariant lift provided by Proposition \ref{prop:step_4}. 
    Since for every $(X,\lambda) \in \put_g$ the measured geodesic lamination $\mu := \mu(X,\lambda)$ given by $(Y,\mu) := \smash{\widehat{F}}(X,\lambda)$ belongs to the projective class of $\lambda \in \ml_g$, we can consider the continuous function $c \colon \put_g \to \mathbf{R}^+$ which to every $(X,\lambda) \in \put_g$ assigns the unique scaling factor $c(X,\lambda) > 0$ such that
    \begin{equation}
    \label{eq:c}
    \mu(X,\lambda) = c(X,\lambda) \cdot \lambda.
    \end{equation}
    Our goal is to show that $c \colon \put_g \to \mathbf{R}^+$ is identically equal to $e^s$. 
    
    Denote by $T_\gamma \in \mcg_g$ the Dehn twist of $S_g$ along a simple closed curve $\gamma$. One can check that, for every $(X,a \cdot \gamma) \in \put_g$ with $a > 0$ and $\gamma$ a simple closed curve on $S_g$, the period of the earthquake flow orbit of $$(X,a \cdot \gamma) \in \put_g/\langle T_\gamma \rangle$$ is exactly $\ell_\gamma(X)/a$. 
    
    Now consider $\lambda = \gamma/\ell_\gamma(X)$ with $\gamma$ a simple closed curve on $S_g$. Note that $(X,\lambda)$ has period $\ell_\gamma(X)^2$ in $\put_g/\langle T_\gamma \rangle$ and $\smash{\widehat{F}}(X,\lambda) = (Y, c(X, \lambda) \lambda)$ has period 
    $$\frac{\ell_\gamma(Y) \ell_\gamma(X)}{c(X,\lambda)}=\frac{\ell_\gamma(X)^2}{c(X,\lambda)^2},$$
    where the last equality uses the fact that $c(X, \lambda) \lambda$ must have length 1 on $Y$. As $\smash{\widehat{F}}$ is $\mcg_g$-equivariant and as $s$-normalizers multiply periods by $e^{-2s}$, it follows %from (\ref{eq:c}) and the previous  computations 
    that 
    $$\frac{\ell_\gamma(X)^2}{c(X,\lambda)^2}=e^{-2s}\ell_\gamma(X)^2.$$
    Hence, $c(X,\gamma/\ell_{\gamma}(X)) = e^s$. As $c \colon \put_g \to \mathbf{R}^+$ is continuous and as points of the form $(X,\gamma/\ell_\gamma(X)) \in \put_g$ with $\gamma$ a simple closed curve on $S_g$ are dense in $\put_g$, this finishes the proof.
\end{proof}

We now prove a loop shrinking property for lifts $\smash{\widehat{F}}$ of $s$-normalizers of the earthquake flow. This property will play a crucial role in the proof of Theorem \ref{theo:newmain}. 

\begin{lemma}
    \label{lem:step_6}
     Let $\smash{\widehat{F}}$ be the lift produced by Proposition \ref{prop:step_4} of an $s$-normalizer $F$ of the earthquake flow. Then, for every $X \in \mathcal{T}_g$ and every simple closed curve $\alpha \in \mathcal{S}_g$, if $(Y,\mu) := \smash{\widehat{F}}(X,\alpha/\ell_\alpha(X))$, then $$\ell_{\beta}(Y) \leq e^{-s}  \ell_\beta(X)$$ for every simple closed curve $\beta \in \mathcal{S}_g$ that can be realized disjointly from $\alpha$, with equality if $\beta=\alpha$.
\end{lemma}

\begin{proof} 
      By Lemma \ref{lem:step_5}, $\mu = e^s \cdot \alpha/\ell_{\alpha}(X)$. It follows that
     \[
     1 = \ell_\mu(Y) = e^s \cdot \ell_\alpha(X)^{-1} \cdot \ell_\alpha(Y).
     \]
     Reorganizing the terms in this equality we deduce
     \[
     \ell_\alpha(Y) = e^{-s} \cdot \ell_\alpha(X).
     \]
     
     Let $\beta \in \mathcal{S}_g$ be a simple closed curve that can be realized disjointly from $\alpha$ and is not equal to $\alpha$. We average $\alpha$ and $\beta$ with appropriate weights to obtain simple closed multi-curves on $S_g$ converging to $\alpha/\ell_\alpha(X)$, with unit length with respect to $X$, and whose corresponding earthquake flow orbits are periodic with explicit periods. Indeed, for every $k \in \mathbf{N}$ consider the positive weights
    \begin{gather*}
        a_k = a_k(X,\alpha,\beta) := \left(\ell_\alpha(X) + k^{-1} \cdot \ell_\alpha(X)^{-1} \cdot \ell_{\beta}(X)^2\right)^{-1}, \\
        b_k = b_k(X,\alpha,\beta) := \left(\ell_\beta(X) + k \cdot \ell_\alpha(X)^2 \cdot \ell_\beta(X)^{-1}\right)^{-1}.
    \end{gather*}
    These choices guarantee that for every $k \in \mathbf{N}$,
    \begin{equation}
    \label{eq:a1}
    \ell_\beta(X)/b_k = k \cdot \ell_\alpha(X)/a_k.
    \end{equation}
    For every $k \in \mathbf{N}$ consider the simple closed multi-curve on $S_g$ given by
    \[
    \gamma_k = \gamma_k(X,\alpha,\beta) := a_k(X,\alpha,\beta)\cdot \alpha + b_k(X,\alpha,\beta)\cdot \beta.
    \]
    Direct computations show that
    $
    \ell_{\gamma_k}(X) = 1
    $ for every $k \in \mathbf{N}$.
    Directly from the definitions one can also check that
    \[
    \lim_{k \to \infty} \gamma_k = \alpha/\ell_{\alpha}(X).
    \]
    For every $k \in \mathbf{N}$ consider $(Y_k,\mu_k) := \smash{\widehat{F}}(X,\gamma_k)$.
    By Lemma \ref{lem:step_5}, $\mu_k = e^s \cdot \gamma_k$ for every $k \in \mathbf{N}$. As $\smash{\widehat{F}}$ is continuous, 
    \begin{equation}
    \label{eq:b1}
    Y = \lim_{k \to \infty} Y_k.
    \end{equation}
    
    Fix $k \in \mathbf{N}$. Denote by $T_\alpha, T_\beta \in \mcg_g$ the Dehn-twists of $S_g$ along $\alpha$ and $\beta$. A direct computation using (\ref{eq:a1}) shows that the earthquake flow orbit of the image of $(X,\gamma_k)$ in $\put_g/\langle T_\alpha,T_\beta \rangle$ is periodic with period given by the least common multiple
    \begin{equation}
    \label{eq:a_2}
    \mathrm{lcm}(\ell_{\alpha}(X)/a_k, \ell_\beta(X)/b_k) = \ell_\beta(X)/b_k.
    \end{equation}
    Analogously, the earthquake flow orbit of the image of $(Y_k,\mu_k)$ in  $\put_g/\langle T_\alpha,T_\beta \rangle$ is periodic if and only if the following least common multiple is finite, in which case it is exactly the period of the orbit,
    \begin{equation}
    \label{eq:a3}
    \mathrm{lcm}(\ell_\alpha(Y_{k})/(e^s \cdot a_k), \ell_\beta(Y_{k})/(e^s \cdot b_k)).
    \end{equation}
    Since $s$-normalizers multiply periods by $e^{-2s}$, for the periods in (\ref{eq:a_2}) and (\ref{eq:a3}) to agree, it is necessary that
    \[
    \ell_\beta(Y_k) \leq e^{-s} \cdot \ell_\beta(X).
    \]
   Taking limits as $k \to \infty$ and using (\ref{eq:b1}) we conclude
    \[
    \ell_\beta(Y) \leq e^{-s} \cdot \ell_\beta(X). \qedhere
    \]
\end{proof}

% We say an orbifold isomorphism $\smash{\widehat{F}} \colon \pum_g \to \pum_g$ is a flip of the earthquake flow if $\smash{\widehat{E}_t} \circ \smash{\widehat{F}} = \smash{\widehat{F}} \circ \smash{\widehat{E}_{-t}}$ for every $t \in \mathbf{R}$. The same arguments used to prove Lemma \ref{lem:step_5} and Proposition \ref{prop:step_6} yield the following results.

% \begin{proposition}
%     \label{prop:step_5_flip}
%     Every flip $\smash{\widehat{F}} \colon \pum_g \to \pum_g$ of the earthquake flow admits a $\mcg_g$-equivariant lift $F \colon \put_g \to \put_g$ such that for every $(X,\lambda) \in \put_g$, if $(Y,\mu) := F(X,\lambda) \in \put_g$, then $\mu = \lambda$.
% \end{proposition}

% \begin{proposition}
%     \label{prop:step_6_flip}
%     Every flip $\smash{\widehat{F}} \colon \pum_g \to \pum_g$ of the earthquake flow admits a $\mcg_g$-equivariant lift $F \colon \put_g \to \put_g$ with the following property: For every $X \in \mathcal{T}_g$ and every simple closed curve $\alpha \in \mathcal{S}_g$, if $(Y,\mu) := F(X,\alpha/\ell_\alpha(X)) \in \put_g$, then $\mu = \alpha/\ell_\alpha(X)$, $\ell_\alpha(Y) = \ell_\alpha(X)$ and $\ell_{\beta}(Y) \leq \ell_\beta(X)$ for every simple closed curve $\beta \in \mathcal{S}_g$ that can be realized disjointly from $\alpha$.
% \end{proposition}

We can now conclude the proof of Proposition \ref{prop:nor} as follows. 

\begin{proof}[Proof of Proposition \ref{prop:nor}]
    Suppose by contradiction that $s \neq 0$. By working with the inverse of $F$ if $s < 0$, we can assume without loss of generality that $s > 0$. Denote by $\smash{\widehat{F}}$ the $\mcg_g$-equivariant lift  provided by Proposition \ref{prop:step_4}. Let $\alpha,\beta,\gamma \in \mathcal{S}_g$ be simple closed curves such that $\alpha$ can be realized disjointly from $\beta$ and $\gamma$, and such that $\beta$ and $\gamma$ have positive geometric intersection number. Fix $X \in \mathcal{T}_g$ and let $$(X_n,\lambda_n) := \smash{\widehat{F}}^n(X,\alpha/\ell_\alpha(X))$$ for every $n \in \mathbf{N}$. By Lemma \ref{lem:step_6}, there exists $N \in \mathbf{N}$ such that $\ell_\beta(X_N)$ and $\ell_\gamma(X_N)$ are arbitrarily small, contradicting the collar lemma for hyperbolic surfaces.
\end{proof}

\section{The centralizer of the earthquake flow}

In this section we show that the extended centralizer of the earthquake flow is trivial. 

\begin{proposition}\label{prop:extendedcentralizer}
$C_\pm(E)=E$.
\end{proposition}

%Recall that an orbifold isomorphism $\widehat{C} \colon \pum_g \to \pum_g$ is a centralizer of the earthquake flow if $\smash{\widehat{E}_t} \circ \smash{\widehat{C}} = \smash{\widehat{C}} \circ \smash{\widehat{E}_{t}}$ for every $t \in \mathbf{R}$. Equivalently, a centralizer of the earthquake flow is a $0$-normalizer of it. Our next goal is to show that the group of centralizers of the earthquake flow is trivial.
We proceed in several steps, starting with the following geometric result.

\begin{lemma}
    \label{lem:nor_part_1}
    Let $X$ and $Y$ be a pair of compact, connected, and orientable diffeomorphic hyperbolic surfaces with at least one totally geodesic boundary component. Suppose that, for some pair of markings on $X$ and $Y$, the lengths of the boundary components of $X$ agree with those of $Y$, and, for every simple closed curve, the length of its geodesic representative on $Y$ is at most the length of its geodesic representative on $X$. Then, $X$ and $Y$ are isometric.
\end{lemma}

An analogous statement for closed surfaces is well known \cite[Theorem 3.1]{Thu98}. We do not know if the exact statement of Lemma \ref{lem:nor_part_1} has appeared before in the literature, but in any case, a short proof is possible from known results. 

\begin{proof}[Proof of Lemma \ref{lem:nor_part_1}]
    The monotonicity of the summands in Mirzakhani's generalized McShane's identity \cite[Theorem 1.3]{Mir07a} guarantees that, if $X$ and $Y$ satisfy the assumptions, then, for every simple closed curve, the lengths of its geodesic representatives on $X$ and $Y$ are equal. As the isometry class of a marked hyperbolic structure with totally geodesic boundary components on a compact, connected, orientable surface is determined by its marked length spectrum\footnote{A proof can be obtained by adapting the arguments in \cite[Proof of Theorem 10.7]{FM11}.}, we conclude that $X$ and $Y$ are isometric.
\end{proof}

 The following result shows that centralizers of the earthquake flow map points of the form $(X,\alpha/\ell_\alpha(X)) \in \pum_g$ into their own earthquake flow orbit. 

\begin{lemma}
    \label{lem:nor_part_2}
    Suppose $F\in C_\pm(E)$ and let $\smash{\widehat{F}}$ be the lift provided by Proposition \ref{prop:step_4}. Then, for every $X \in \mathcal{T}_g$ and every simple closed curve $\alpha \in \mathcal{S}_g$, there exists a unique $t \in \mathbf{R}$ satisfying
    \[
    \smash{\widehat{F}}(X,\alpha/\ell_\alpha(X)) = E_t(X,\alpha/\ell_\alpha(X)).
    \]
\end{lemma}

For the proof it is helpful to recall that an element of the extended centralizer is nothing other than an $s$-normalizer with $s=0$. 

\begin{proof}[Proof of Lemma \ref{lem:nor_part_2}]
 Let $(Y,\mu) := \smash{\widehat{F}}(X,\alpha/\ell_{\alpha}(X)) \in \put_g$. Lemmas  \ref{lem:step_5} and  \ref{lem:step_6}  ensure that $\mu = \alpha/\ell_\alpha(X) \in \ml_g$, $\ell_\alpha(Y) = \ell_\alpha(X)$, and $\ell_\beta(Y) \leq \ell_\beta(X)$ for every simple closed curve $\beta \in \mathcal{S}_g$ that can be realized disjointly from $\alpha$.
 
 Cutting $X$ and $Y$ along the corresponding geodesic representatives of $\alpha$ on each surface yields a pair of (possibly disconnected) hyperbolic surfaces with totally geodesic boundary components of matching lengths. 
 Lemma \ref{lem:nor_part_1} guarantees these surfaces are isometric. As $X$ and $Y$ can be recovered from isometric pieces by glueing along the boundary components corresponding to $\alpha$, we deduce that $X$ and $Y$ only differ by a Fenchel-Nielsen twist along $\alpha$. In other words, 
 \[
    \smash{\widehat{F}}(X,\alpha/\ell_\alpha(X)) = (Y,\mu) = E_t(X,\alpha/\ell_\alpha(X)). \qedhere
 \]
\end{proof}

 The following result extends the conclusion of Lemma \ref{lem:nor_part_2} to arbitrary points $(X,\lambda) \in \put_g$.

\begin{lemma}
    \label{lem:nor_part_3}
    Suppose $F\in C_\pm(E)$ and let $\smash{\widehat{F}}$ be the lift provided by Proposition \ref{prop:step_4}.
    Then there exists a continuous, $\mcg_g$-invariant function $t \colon \put_g \to \mathbf{R}$ such that for every $(X,\lambda) \in \put_g$, $t = t(X,\lambda)$ satisfies
    \begin{equation*}
        \smash{\widehat{F}}(X,\lambda) = E_{t}(X,\lambda)
    \end{equation*}
    and is the unique real number satisfying this equation. 
    
    Furthermore, if $F\in C(E)$, then $t$ is earthquake flow invariant, and, if $F\in C_\pm(E)\setminus C(E)$, then $T$ is ``twisted-equivariant'' in the sense that $$t(E_s(X,\lambda)) = t(X,\lambda)-2s.$$
\end{lemma}

\begin{proof}
     Fix $(X,\lambda) \in \put_g$. As weighted simple closed curves are dense in $\ml_g$, one can find a sequence $(\lambda_n)_{n \in \mathbf{N}}$ of  length 1  weighted simple closed curves such that $\lambda_n \to \lambda$ in $\ml_g$ as $n \to \infty$. By Lemma \ref{lem:nor_part_2}, for every $n \in \mathbf{N}$ there exists $t_n \in \mathbf{R}$ such that 
    \begin{equation}
    \label{eq:zz}
    \smash{\widehat{F}}(X,\lambda_n) = E_{t_n}(X,\lambda_n).
    \end{equation}
    
    \begin{claim}
    The sequence $(t_n)_{n \in \mathbf{N}}$ is bounded. 
    \end{claim}
    
    \begin{proof} Suppose by contradiction this was not the case. Assume $t_n$ diverges to $+\infty$ along a subsequence; the case when $t_n$ diverges to $-\infty$ along a subsequence can be treated in an analogous way. Rename this subsequence as $(t_n)_{n \in \mathbf{N}}$ and assume without loss of generality that all of its terms are positive. Let $\mu \in \ml_g$ be a measured geodesic lamination such that
    \begin{equation}
    \label{eq:der}
    \iint_X \cos \theta \thinspace d\lambda \thinspace d\mu > 0,
    \end{equation}
    where $\theta$ is the angle measured counterclockwise from $\mu$ to $\lambda$ at each intersection between $\mu$ and $\lambda$. The existence of such a measured geodesic lamination $\mu \in \mathcal{ML}_g$ can be argued as follows. By the infinitesimal version of Thurston's earthquake theorem (see for instance \cite[Appendix, Theorem 2]{Ker83}), every tangent vector at $X \in \mathcal{T}_g$ can be realized by an infinitesimal earthquake. In particular, by Kerckhoff's derivative formula \cite[Corollary 3.4]{Ker83}, the only way $\mu$ could not exist is if the function $Y \in \mathcal{T}_g \mapsto \ell_\lambda(Y) > 0$ had a critical point, and, by convexity of length functions \cite[Section 3, Theorem 1]{Ker83}, a minimum at $X$. This is not possible, as can be seen, for instance, using shear coordinates and reverse stretch lines.
    
    By work of Kerckhoff \cite[Corollary 3.4]{Ker83}, the integral in \eqref{eq:der} is equal to the derivative at $t = 0$ of the convex function $t \mapsto \ell_\mu(E_t(X,\lambda))$. By continuity, there exists $c > 0$ and $N \in \mathbf{N}$ such that for every $n \geq N$,
    \[
    \iint_X \cos \theta \thinspace d\lambda_n \thinspace d\mu > c.
    \]
    Kerckhoff's work guarantees that, for every $n \geq {N}$,
    \begin{equation}
    \label{eq:zzz}
    \ell_\mu(E_{t_n}(X,\lambda_n)) \geq \ell_\mu(X) + t_n \cdot c.
    \end{equation}
    
    Denote by $\pi \colon \put_g \to \mathcal{T}_g$ the natural projection defined by $\pi(X,\lambda) = X$.
    By definition, $$E_{t_n}(X,\lambda_n) = \smash{\widehat{F}}(X,\lambda_n) \in \smash{\widehat{F}}(\pi^{-1}(X)).$$ As $\smash{\widehat{F}}$ is continuous, the set $\smash{\widehat{F}}(\pi^{-1}(X)) \subseteq \put_g$ is compact. It follows that the sequence $(\ell_\mu(E_{t_n}(X,\lambda_n)))_{n \in \mathbf{N}}$ must be bounded. Taking limits as $n \to \infty$ in (\ref{eq:zzz}) yields a contradiction, concluding the proof of the claim. 
    \end{proof}
    
    As $(t_n)_{n \in \mathbf{N}}$ is bounded, it admits a subsequence converging to some $t \in \mathbf{R}$. Taking limits in (\ref{eq:zz}) along this subsequence we deduce
    \begin{equation}
    \label{eq:zzzz}
    \smash{\widehat{F}}(X,\lambda) = E_t(X,\lambda).
    \end{equation}
    The uniqueness of $t \in \mathbf{R}$ satisfying this condition follows directly from the fact that earthquake flow orbits in $\put_g$ are embedded. The continuity of the corresponding function $t \colon \put_g \to \mathbf{R}$ follows from (\ref{eq:zzzz}) and uniqueness. The $\mcg_g$-invariance of $t$ can be verified using (\ref{eq:zzzz}) and the fact that $\smash{\widehat{F}}$ is $\mcg_g$-equivariant. The earthquake flow invariance or twisted-equivariant of $t$ can be verified directly from (\ref{eq:zzzz}) and the fact that $\smash{\widehat{F}}$ is in the extended centralizer of the earthquake flow.
\end{proof}

We are now ready to conclude.

\begin{proof}[Proof of Proposition \ref{prop:extendedcentralizer}]
    Consider the function  $t:\put_g\to \mathbf{R}$ above. Since it is $\mcg_g$-equivariant, it induces a function $t:\pum_g\to \mathbf{R}$. 
    
    If $F\in C(E)$, the function $t$ is earthquake flow invariant. 
    As the earthquake flow on $\pum_g$ is ergodic with respect to a measure of full support, $t$ is equal to a constant $t_0 \in \mathbf{R}$ on a dense set of $\pum_g$. Applying continuity and density we conclude $F = E_{t_0}$ as desired.
    
    Suppose $F\in C_\pm(E)\setminus C(E)$. There exists $c$ such that  the set $t^{-1}((c, c+2))$ has positive measure. The twisted-equivariance gives that for all $k$,  $E_{k}$ maps 
    $t^{-1}((c, c+2))$ into $t^{-1}((c-2k, c+2-2k))$. For different $k$ integral the sets $t^{-1}((c-2k, c+2-2k))$ are disjoint, and since earthquake flow is measure preserving they all have the same measure. So considering all $k$ integral contradicts the fact that the space has finite measure, showing that such an $F$ cannot exist.
\end{proof}

We are now ready to prove that the earthquake flow is strongly asymmetric.

\begin{proof}[Proof of Theorem \ref{theo:newmain}]
Proposition \ref{prop:nor} shows that $N(E)=C_\pm(E)$ and Proposition \ref{prop:extendedcentralizer} shows that $C_\pm(E)=E$.
\end{proof}

\appendix
\section{Minimal sets}\label{A:SW}

In this  appendix we sketch, for the convenience of the reader, a proof of Theorem \ref{theo:min}. The corresponding result in the case of the Teichm\"uller horocycle flow is discussed in detail in \cite{SW04}, and Smillie and Weiss remark there that ``an analogous result for the earthquake flow may be proved by a similar argument.'' Our starting point is the following observation, the details of whose proof are left to the reader. 

\begin{lemma}\label{L:constant}
If $K \subset \pum_g$ is a minimal set for the earthquake flow, and $(X, \lambda)$ and $(X', \lambda')$ are in $K$, then $X-\lambda$ is isometric to $X'-\lambda'$. 
\end{lemma}

\begin{proof}[Sketch of proof.]
For any fixed $(X, \lambda)\in K$, consider the set $K'\subseteq K$ of all $(X', \lambda')\in K$ for which there exists an isometric embedding 
$$X-\lambda \hookrightarrow X'-\lambda'$$
of complementary regions. Since the complementary regions are not changed by the earthquake flow, $K'$ is invariant. A limit argument shows that $K'$ is closed, so the definition of minimality guarantees $K'=K$. 

Thus, for every $(X, \lambda), (X', \lambda')\in K$, each complementary region embeds isometrically into the other. Hence $X-\lambda = X'-\lambda'$.
\end{proof}

We also need the following non-trivial result. 

\begin{proposition}\label{P:key}
If $\lambda$ is not a multi-curve and the orbit of $(X,\lambda)$ is bounded in $\pum_g$, then the orbit accumulates on some $(X', \lambda')$ with $X-\lambda \neq X'-\lambda'$. 
\end{proposition}

In fact, experts believe the following stronger statement is true (and accessible).   

\begin{problem}
Prove that if $\lambda$ is not a multi-curve, then the earthquake flow orbit of $(X,\lambda)$ is not bounded. 
\end{problem}

We will not consider this problem here since it is certainly harder than what we require. The analogous problem for the Teichmüeller horocycle flow is item (IV) in the list of problems at the end of \cite{SW04} and has been considered in unpublished work of Smillie and Weiss.  

Before addressing Proposition \ref{P:key}, we note it implies Theorem \ref{theo:min}.

\begin{proof}[Proof of Theorem \ref{theo:min} assuming Proposition \ref{P:key}.]
If $K$ is a compact minimal set and $(X, \lambda) \in K$, then Lemma \ref{L:constant} implies that for any $(X', \lambda')$ in the orbit closure of $(X,\lambda)$ has $X-\lambda = X'-\lambda'$, and so Proposition \ref{P:key} implies $\lambda$ is a multi-curve. 

The converse implication that if $\lambda$ is a multi-curve then the orbit closure of $(X,\lambda)$ is a minimal set is well known. Indeed, if $T \subseteq \pum_g$ is the subset obtained by starting at $(X,\lambda)$ and independently twisting at each component of $\lambda$, then $T$ is an invariant torus and the earthquake flow is continuously conjugate to a straight line flow on $T$. The converse implication follows from the fact that, for straight lines flows on tori, every orbit closure is a minimal set. 
\end{proof}

We conclude by briefly sketching how the ideas of Smillie-Weiss apply to Proposition \ref{P:key}. Most of the work is divided into two lemmas. 

\begin{lemma}\label{L:FindRectangle}
Suppose $\lambda$ is a measured geodesic lamination on $X$ that is not a multi-curve. Then there exists some $\delta>0$ such that for all $\epsilon>0$ we can find segments $\gamma_1$ and $\gamma_2$ of leaves of $\lambda$ that stay within distance $1$ of each other, and such that all leaves of $\lambda$ that come within $\delta$ of the start point $p_1$ of $\gamma_1$ do so on the  side of $\gamma_1$  containing $\gamma_2$, and all leaves that come within $\delta$ of the end point $p_2$ of $\gamma_2$ do so on the  side of $\gamma_2$ containing $\gamma_1$, and such that the transverse measure of a segment from $\gamma_1$ to $\gamma_2$ is less than $\epsilon$. Moreover, $\gamma_1$ and $\gamma_2$ can be taken to lie on non-isolated leaves of $\lambda$. 
\end{lemma}

In particular, it follows that both $\gamma_i$ are segments of leaves of $\lambda$ adjacent to regions of $X-\lambda$. See Figure \ref{F:gamma12}. 

\begin{figure}[h!]\centering
\includegraphics[width=0.6\linewidth]{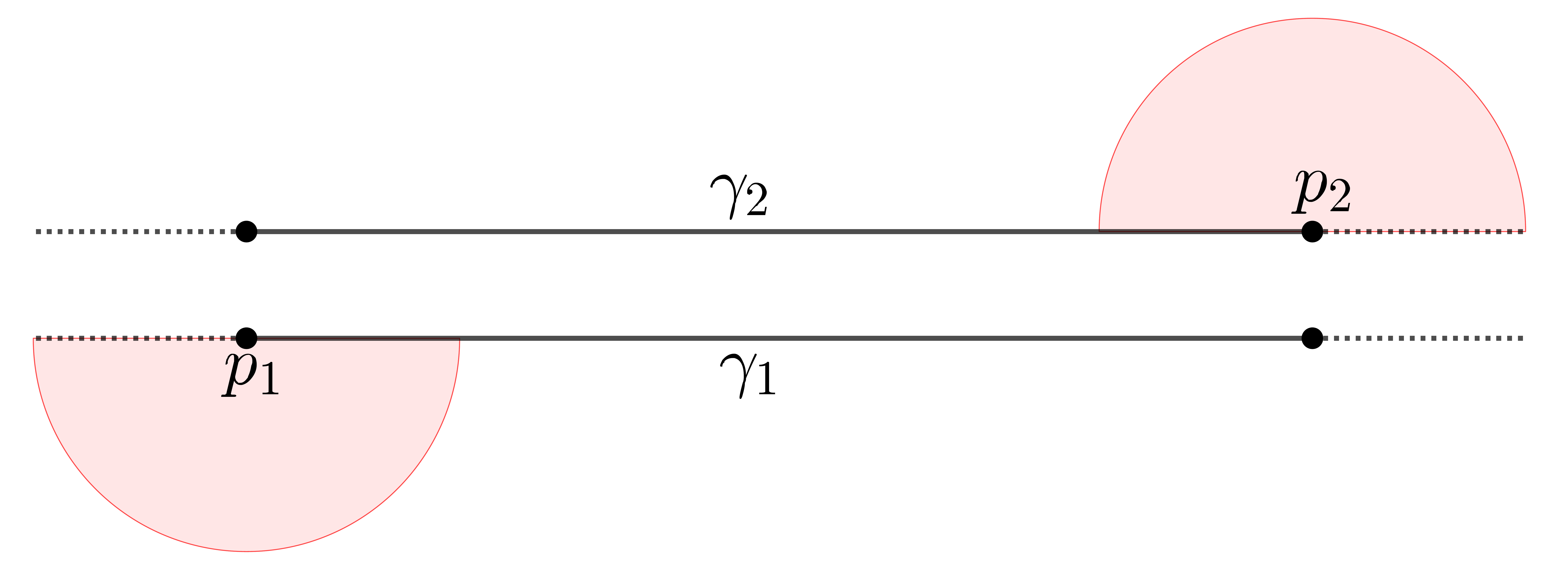}
\caption{An illustration of the $\gamma_i$. The red half balls of radius $\delta$ do not intersect $\lambda$.}
\label{F:gamma12}
\end{figure}

The proof will use the concept of the thick part of a surface with boundary, which can be defined by embedding the surface in its double and taking the thick part there; see for example \cite[Section 2.1]{LW} for details. 

\begin{proof}[Sketch of proof.]
Without loss of generality assume $\gamma$ has no closed leaves. Start with $p_1$ on the boundary of the thick part of $X-\lambda$, on a leaf $\alpha$ of $\lambda$. Pick a point $q$ that is very close to $p_1$ and on a leaf $\beta$ of $\lambda$. Follow both leaves $\alpha, \beta$ in the same direction until they are distance $1/10$ apart. The region $R$ between these segments of $\alpha$ and $\beta$, illustrated in Figure \ref{F:R}, has definite area. 

\begin{figure}[h]\centering
\includegraphics[width=0.5\linewidth]{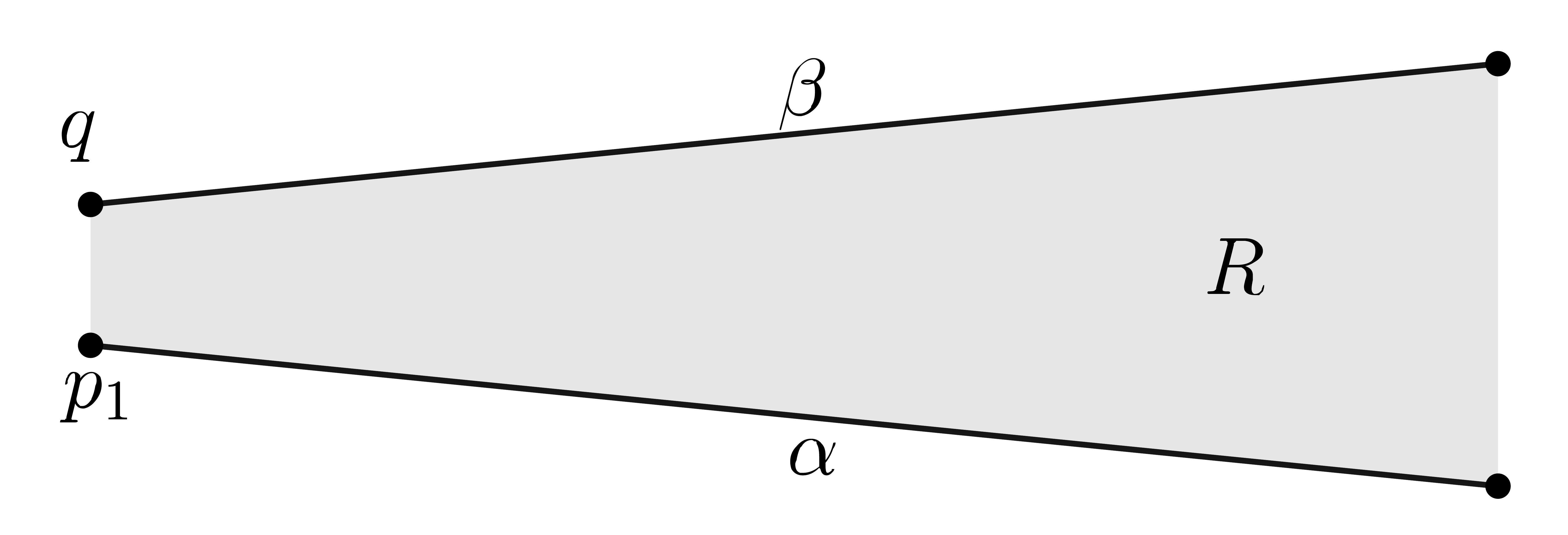}
\caption{An illustration of the region $R$}
\label{F:R}
\end{figure}

The area of the thin part of $X-\lambda$ is small, so the thick part must intersect $R$. (Here the thick part should be defined appropriately using $\delta$ and $\delta$ should be taken small enough.)

We then pick $p_2$ to be on the boundary of the thick part of $X-\lambda$ intersected with $R$. (One should pick $p_2$ so that the thick part and $\alpha$ are on different sides of the leaf through $p_2$.) We define $\gamma_1$ to be the segment of $\alpha$ from $p_1$ to the projection of $p_2$ onto $\alpha$, and similarly define $\gamma_2$ using the leaf through $p_2$. 
\end{proof}

\begin{lemma}\label{L:ShearRectangle}
There exists a universal constant $C>0$ such that the following holds. Consider any measured geodesic lamination on $\mathbf{H}$, any segments $\gamma_1, \gamma_2$ of non-atomic leaves of $\lambda$ that stay within distance $1$ of each other, and any $p_1\in \gamma_1$, $p_2 \in \gamma_2$. Assume there are leaves of $\lambda$ that go between $p_1$ and $p_2$. Let $\lambda_{max}$ be a maximal geodesic lamination containing $\lambda$. Assume the $p_i$ lie on the boundary of $\mathbf{H}-\lambda_{max}$. Then there is a unique $t \in \mathbf{R}$ such that the image of $p_1$ and $p_2$ under the time $t$ earthquake of $\lambda$ can be joined by a segment $s$ of a leaf of the horocyclic foliation of $\lambda_{max}$ and this segment has length at most $C$. 
\end{lemma} 

In applications, often $\lambda$ is already maximal, so $\lambda_{max}=\lambda$. The main conclusion here is that $p_1$ and $p_2$ become bounded distance from each other; the use of the horocyclic foliation (and $\lambda_{max}$) is merely a convenient technical tool to obtain this.

One should of course think of $\mathbf{H}$ as the universal cover of a closed surface $X$; we use the universal cover only so that we do not have to specify a homotopy class for the arc $s$.

\begin{proof}[Sketch of proof.]
The first claim is related to the fact that shears change linearly under earthquakes; see for example the survey \cite[Section 4]{Wri18}. 

If one considers a rectangle $R$ bounded by $\gamma_1$ and $\gamma_2$, then $\lambda$ divides this rectangle up into countably many small rectangles bounded by leaves of $\lambda$. The preimage of $s$ on $(X,\lambda)$ consists of one horocyclic arc in each small rectangle; compare to a Cantor staircase.  

For each small rectangle, one can define its maximum height to be the maximum length of a horocyclic arc crossing that rectangle. A standard estimate shows that the sum of the maximum heights is at most some constant $C$; see \cite[page 16]{Thu98}. This uses the fact that the $\gamma_i$ remain within distance 1 of each other. 

The length of $s$ is the sum of the lengths of the horocyclic arcs of $E_{-t}(s)$, which is at most $C$. This gives the result. 
\end{proof}

\begin{proof}[Sketch of proof of Proposition \ref{P:key}.]
Consider a sequence $\epsilon_n\to 0$  and for each $n \in \mathbf{N}$ let $\gamma_{1,n}, \gamma_{2,n}, p_{1,n}, p_{2,n}$ be as provided by Lemma \ref{L:FindRectangle} with $\epsilon=\epsilon_n$.  

The output of Lemma \ref{L:ShearRectangle} is a sequence of points $(X_n, \lambda_n) \in \pum_g$ on the earthquake flow orbit of $(X,\lambda)$ such that two points on the boundary of the thick part of $X_n-\lambda_n=X-\lambda$ are joined by a path on $X_n$ of hyperbolic length at most $C$ and transverse measure going to $0$ as $n \to \infty$. By extending these paths into the thick part and taking geodesic representatives, we obtain geodesic paths $\sigma_n$ on $X_n$ of lengths bounded above and below, which are uniformly transverse to $\lambda_{n}$, and which have the same transverse measureas the original paths of length at most C. 

Passing to a subsequence if necessary, we can assume $(X_n, \lambda_n)$ converge to some $(X_\infty, \lambda_\infty) \in \pum_g$. For convenience, we can also assume that the supports of the $\lambda_n$ converge to a geodesic lamination $\smash{\widehat{\lambda}_\infty}$ which contains the support of $\lambda_\infty$.

Since the complementary regions $X_n-\lambda_n$ are constant, it follows that $X-\lambda = X_\infty - \smash{\widehat{\lambda}_\infty}$. Thus, to show that  $X-\lambda\neq X_\infty-\lambda_\infty$, it suffices to show that some leaves of the geodesic lamination $\smash{\widehat{\lambda}_\infty}$ are not contained in the support of ${\lambda}_\infty$.

This is verified by considering a limit $\sigma$ of the geodesic segments $\sigma_n$; the limit $\sigma$ has length bounded above and below, is transverse to $\smash{\widehat{\lambda}_\infty}$, and has $0$ transverse measure with respect to $\lambda_\infty$.
\end{proof}


\bibliographystyle{amsalpha}


\bibliography{bibliography}

\end{document}